\documentclass{amsart}

\usepackage{amsmath,amsthm}
\usepackage{amsfonts,amssymb}

\newtheorem{thm}{Theorem}[section]
\newtheorem{cor}[thm]{Corollary}
\newtheorem{lem}[thm]{Lemma}

\newtheorem{defn}[thm]{Definition}

\theoremstyle{remark}

 \def\wb{{\mathbf w}}
 \def\xb{{\mathbf x}}
 \def\yb{{\mathbf y}}

 \def\Kb{{\mathbf K}}
 \def\Pb{{\mathbf P}}

 \def\CH{{\mathcal H}}

 \def\CL{{\mathcal L}}
 
 \def\CP{{\mathcal P}}

 \def\CV{{\mathcal V}}

 \def\RR{{\mathbb R}}

\begin{document}
\title[]{On hyperinterpolation on the unit ball} 
 
\author{Jeremy Wade}
\address{Department of Mathematics\\ Pittsburg State University\\
    Pittsburg, KS  66762.}
\email{jwade@pittstate.edu}

\date{\today}
\thanks{The author wishes to thank Yuan Xu for his helpful comments.}

\begin{abstract}
	We prove estimates on the Lebesgue constants of the hyperinterpolation operator for functions on the unit ball $B^d \subset \RR^d$, with respect to Gegenbauer weight functions, $(1-\|\xb\|^2)^{\mu-1/2}$.  The relationship between orthogonal polynomials on the sphere and ball is exploited to achieve this result, which provides an improvement on known estimates of the Lebesgue constant for hyperinterpolation operators on $B^2$.
\end{abstract}

\maketitle

\section{Introduction}
Hyperinterpolation was introduced by Sloan in \cite{Sloan95} as an approximation technique which preserves a classical inequality of Erd\"os and Tur\`an, concerning interpolation polynomials in a univariate setting, to higher dimensions.  Combining two common approximation techniques, hyperinterpolation uses the Fourier orthogonal expansion of a function in orthogonal polynomials, but approximates the integrals used in the expansion by means of a cubature formula with positive weights.  Hence, hyperinterpolation provides a polynomial approximation which relies only on a discrete set of data.

Much attention has been given to hyperinterpolation on spheres and balls.  On spheres, Le Gia and Sloan \cite{LeGiaSloan} studied the Lebesgue constant of the hyperinterpolation operator, showing that Lebesgue constant grew on the order of the Lebesgue constant of the minimal projection operator on $C(S^{d-1})$, subject to certain ``regularity'' conditions on the cubature formulae.  Reimer was able to independently prove in \cite{Reimer00} the same estimate obtained by Le Gia and Sloan, and further, that the ``regularity'' condition imposed on the cubature formulae was not needed, and that the result held for any positive cubature formula.  Specifically, if $\CL_n$ is a hyperinterpolation operator on the space of continuous functions on $S^{d-1}$, which is a projection onto the space of polynomials of degree $n$, Le Gia and Sloan, and Reimer, were able to prove there exist postive constants $b_d$ and $c_d$ satisfying
\begin{displaymath}
	b_d n^{(d-2)/2} \leq \| \CL_n \| \leq c_d n^{(d-2)/2}.
\end{displaymath}
Similar estimates were obtained for Sobolev spaces on the sphere in \cite{HesseSloan}.  

Regarding hyperinterpolation on the ball, Atkinson, Chien, and Hansen \cite{Atkinson09} investigated hyperinterpolation on continuous functions on the unit disk $B^2$ in $\RR^2$.  An esimate for the Lebesgue constant for a hyperinterpolation operator using a particular cubature formula was found.  The cubature formula was obtained by combining quadrature formulas on a closed interval and a circle, and the estimate on the Lebesgue constant obtained for the hyperinterpolation operator was
\begin{displaymath}
	\|\CL_n\| \leq c\, n \log(n+1)
\end{displaymath}
for some positive constant $c$. 

In this paper, we improve the estimate obtained in \cite{Atkinson09}, and extend the results to general hyperinterpolation formulas stemming from cubature formulas with positive weights on balls in higher dimensions, where the cubature formulas may also be with respect to certain Gegenbauer weight functions.  Our results rely on the results on the sphere obtained in \cite{Reimer00}, and the connection between orthogonal polynomials on the sphere and the ball, which were developed by Xu.  

In Section 2 of this paper, we will cover the background information needed for the proof of our theorem.  In Section 3, we will state and prove our main theorem on the Lebesgue constants of the hyperinterpolation operators.

\setcounter{equation}{0}

\section{Orthogonal Polynomials and Cubature Formulae} 
\setcounter{equation}{0}

\subsection{Orthogonal Polynomials on $B^d$ and $S^d$}

The structural interplay between orthogonal polynomials on the unit ball $B^d$ and $S^d$ are essential in our result, so we review these results here.  It is assumed that the reader has a basic understanding of orthogonal polynomials; the standard references are \cite{Szego} for univariate and \cite{DX} for multivariate orthogonal polynomials.  

We denote the space of polynomials in $d$ variables with total degree less than or equal to $n$ by $\Pi_n^d$.  We define $\CP_n^d$ to be the space of polynomials of total degree $n$ in $d$ variables. 

We denote the space of orthogonal polynomials of total degree $n$ in $d$ variables with respect to the weight function $\omega$ on a set $\Omega$ by $\CV_n^d(\omega)$.  We will be particularly interested in orthogonal polynomials on $B^d$, with respect to the Gegenbauer weight function, $\wb_\mu (\xb) = (1-\|\xb\|^2)^{\mu - 1/2}$ for $\mu>-1/2$.

The dimension of the space $\CV_n^d(\wb_\mu)$ is $r_n^d = \binom{n+d-1}{d}$.  Given a orthonormal basis of $\CV_n^d(\wb_\mu)$, $\{P_{n,1}^\mu, P_{n,2}^\mu, \ldots , P_{n, r_n^d}^\mu\}$, the reproducing kernel of $\CP_n^d$, $\Pb_n(\wb_\mu; \xb, \yb)$, is the function defined by
\begin{equation}\label{eqn:repker}
	\Pb_n(\wb_\mu; \xb, \yb) = \sum_{j=0}^{r_n^d} P_{n,j}^\mu(\xb) P_{n,j}^\mu(\yb).
\end{equation}
It earns its name because it ``reproduces'' orthogonal polynomials of degree $n$:
\begin{displaymath}
	f(\xb)=\int_{B^d}\, \Pb_n(\wb_\mu; \xb,\yb) f(\yb) \wb_\mu(\yb) d\yb 
\end{displaymath}
for all $f \in \CP_n^d$. 
The reproducing kernel of the space $\Pi_n$ with respect to $\wb_\mu$ is the function $\Kb_n(\wb_\mu; \xb, \yb)$, defined by
\begin{displaymath}
	\Kb_n(\wb_\mu; \xb, \yb) = \sum_{j=0}^n \Pb_j(\wb_\mu; \xb, \yb)
\end{displaymath}
and satisfying
\begin{equation} \label{eqn:ballrepker}
	f(\xb)=\int_{B^d}\, \Kb_n(\wb_\mu; \xb,\yb) f(\yb) \wb_\mu(\yb) d\yb 
\end{equation}
for all $f \in \Pi_n^d$. 

We now turn to orthogonal polynomials on the sphere.  Orthogonal polynomial structure on the unit sphere differs somewhat from the orthogonal polynomial structures observed in Euclidean regions with non-zero measure, due to the fact that unit sphere is a algebraic surface.  The orthogonal polynomials on the unit sphere with respect to the surface measure $d\omega$ ($d\omega_d$ for $S^d$ when dimension requires specification) are known as spherical harmonics, which are harmonic, homogeneous polynomials restricted to the unit sphere.  It turns out that harmonic, homogenuous polynomials of different degrees in $d+1$ variables are orthogonal with respect to integration over $S^d$.  We denote the space of homogeneous, harmonic polynomials of degree $n$ in $d+1$ variables by $\CH_n^{d+1}$.  When restricted to the unit sphere, the space of polynomials of degree $n$ can be decomposed orthogonally into the subspaces $\CH_k^{d+1}$, where we are considering the polynomials to be restricted to the unit sphere.

Spherical harmonics can be generalized in order to include the possibility of weight functions; for example, a study of spherical harmonics with respect to weight functions invariant under reflection groups is given in Chapter 5 of \cite{DX}.  When considering spherical harmonics with respect to a given weight function $H$, we use the notation $\CH_n^{d+1}(H)$.  

A reproducing kernel for spherical harmonics also exists.  We may take an orthonormal basis for the space of spherical harmonics of degree $n$ to obtain the reproducing kernel for that space, in the same way as \eqref{eqn:repker}. We denote the reproducing kernel of the space of spherical harmonics of degree $n$ with respect to the weight function $H$ by $P_n(H; \xb, \yb)$, and it satisfies
\begin{equation}
	\label{eqn:sphharrepker}
	f(\xb)=\int_{S^{d}}\, P_n(H; \xb, \yb) f(\yb) H(\yb) d \omega_d(\yb)
\end{equation}
where $\omega_d$ is the standard surface measure for $S^{d}$, whenever $f$ is a spherical harmonic of degree $n$.  We may also obtain a reproducing kernel for the space of all polynomials, restricted to the unit sphere, of degree less than or equal to $n$, $K_n$, by letting $K_n(H; \xb, \yb) = \sum_{k=0}^n P_k(H; \xb, \yb)$; $K_n$ then satisfies 
\begin{equation*}
	f(\xb)=\int_{S^{d}}\, K_n(H; \xb, \yb) f(\yb) H(\yb) d \omega_d(\yb)
\end{equation*}
whenever $f$ is a polynomial of degree less than or equal to $n$ and $\xb \in S^d$.

\subsection{Relationship between Orthogonal Polynomials on Spheres and Balls}

The relationships between orthogonal polynomials on spheres and balls is very rich.  We will focus on the relationship between reproducing kernels on spheres and balls.  This subsection outlines this relationship, which is given detailed treatment in Section 3.8 in \cite{DX}.  We first provide some definitions that are needed to describe the relationship we will exploit.

\begin{defn}
	A function $H$ is \underline{centrally symmetric} if $H(\xb)=H(-\xb)$ for all $\xb$ in its domain.  
\end{defn}

\begin{defn}
	A function $H$ is \underline{positively homogeneous of order $\sigma$} if, for all $\xb$ in its domain and all $t>0$, $H(t\xb)=t^\sigma H(\xb)$.
\end{defn}

\begin{defn}\label{defn:admiss}
	A function $H$ is \underline{admissible} on $\RR^{d+m+1}$ if $H(\xb)=H_1(\xb_1) H_2(\xb_2)$, where $\xb_1 \in \RR^d$,  $\xb_2 \in \RR^{m+1}$, $H_1$ is centrally symmetric, and $H_2$ is positively homogeneous.
\end{defn}

With these definitions, we are able to state the following theorem, which will be crucial in our results.

\begin{thm}[Thm 3.8.12, \cite{DX}]
	Let $d$ and $m$ be positive integers.  Suppose that $H$ is an admissible weight function on $\RR^{d+m+1}$ as in Definition \ref{defn:admiss}, where $H_2$ is positively homogeneous of order $2\tau$.  Let $W_H^m(\yb)=H_1(\yb) \wb_{\tau+m/2}(\yb)$ be defined on $B^d$.  Then the reproducing kernel of $\CP_n$ on $B^d$ with respect to $W_H^m$, $\Pb_n (W_H^m; \cdot, \cdot)$, is related to the reproducing kernel of spherical harmonics on $S^{d+m}$ with respect to the weight function $H$, $P_n(H; \cdot, \cdot)$ by the formula
	\begin{equation}
		\Pb_n(W_H^m; \xb, \yb)= \int_{S^m} P_n (H; (\xb, \xb^c), (\yb, \sqrt{1-\|\yb\|^2}\eta)) H_2(\eta) d\omega_m(\eta),
		\label{eqn:repker1}
	\end{equation}
	where $\xb^c \in \RR^{m+1}$ is any point satisfying $(\xb, \xb^c) \in S^{d+m+1}$.  In particular, if $H_2=H_1=1$, then  
	\begin{equation}
		\Pb_n( \wb_{m/2}; \xb, \yb)= \int_{S^m} P_n (1; (\xb, \xb^c), (\yb, \sqrt{1-\|\yb\|^2}\eta)) d\omega_m(\eta).
		\label{eqn:repker1.1}
	\end{equation}
	\label{thm:xuker}
\end{thm}

We may sum both sides of the formula in \eqref{eqn:repker1} to obtain the following corollary.

\begin{cor}
	Under the assumptions of Theorem \ref{thm:xuker}, we have 
	\begin{equation}
		\Kb_n(W_H^m; \xb, \yb)= \int_{S^m} K_n (H; (\xb, \xb^c), (\yb, \sqrt{1-\|\yb\|^2}\eta)) H_2(\eta) d\omega_m(\eta),
		\label{eqn:repker2}
	\end{equation}
	where $\xb^c$ has the same meaning as in Theorem \ref{thm:xuker}.  If $H_1=H_2=1$, then
	\begin{equation}
		\Kb_n( \wb_{m/2}; \xb, \yb)= \int_{S^m} K_n (1; (\xb, \xb^c), (\yb, \sqrt{1-\|\yb\|^2}\eta)) d\omega_m(\eta).
		\label{eqn:repker2.1}
	\end{equation}
	\label{cor:repkerballtosphere}
\end{cor}

\subsection{Cubature and Hyperinterpolation}

Using the reproducing kernels, one is able to obtain the best polynomial approximation in the $L^2$ norm to a given function.  Hyperinterpolation uses a cubature formula to approximate the integrals required to obtain this approximation.

\begin{defn} 
	Given a subset $\Omega$ of $\RR^d$ with weight function $\omega$, a \underline{cubature formula} of degree $n$ is a set of positive weights $(\lambda_\alpha) \subset \RR^+$ and nodes $(x_\alpha) \subset \Omega$, where $\alpha$ belongs to some index set $A$, satisfying 
\begin{displaymath}
	\sum_{\alpha \in A} \lambda_\alpha f(x_\alpha) = \int_{\Omega} f(x) \omega(x) dx
\end{displaymath}
whenever $f$ is a polynomial of degree less than or equal to $n$.
\end{defn}

We will make use of an arbitrary cubature on $S^d$ in our proof.  Many examples of such cubature formulae are known.  For examples, see \cite{LuoMeng}, \cite{Stroud} and Theorem 5.3 in \cite{Xu01}.  

Given a cubature formula on $B^d$ or $S^d$, the hyperinterpolation operator is obtained by approximating the integral in the formula \eqref{eqn:ballrepker} or \eqref{eqn:sphharrepker} with an appropriate cubature formula.  Specifically, taking a cubature formula of degree $n$, then the linear operator
\begin{displaymath}
	\CL_n(f)(\xb) := \sum_{\alpha \in A} \lambda_\alpha K_n(\xb_\alpha, \xb)f(\xb_\alpha)
\end{displaymath}
satisfies
\begin{displaymath}
	\CL_n(f)(\xb) = f(\xb)
\end{displaymath}
for all $f \in \Pi_n$.  The operator $\CL_n$ is the \emph{hyperinterpolation operator of degree $n$}.

\section{Main Results}

The main result of this paper is the following theorem, concerning the Lebesgue constant of $\CL_n$, which is defined by
\begin{displaymath}
	\| \CL_n\| = \sup \{ \|\CL_n f\|_{C(B^d)} \, : \, f \in C(B^d),\, \|f\|_{C(B^d)} \leq 1 \}.
\end{displaymath}

\begin{thm}
	Let $\CL_{2n}$ be the hyperinterpolation operator arising from a cubature formula with positive weights on the ball $B^d$, $d \geq 1$, with respect to the Gegenbauer weight function $\wb_\mu(\xb)=(1-\|\xb\|^2)^{\mu-1/2}$, where $\mu$ is a positive half-integer.  Then there exist positive constants $b_{d,\mu}$ and $c_{d,\mu}$ satisfying
	\begin{displaymath}
		b_{d,\mu} n^{(d-1)/2+\mu}\leq\|\CL_{2n}\| \leq c_{d,\mu} n^{(d-1)/2 + \mu}.
	\end{displaymath}
	\label{thm:main}
\end{thm}

Integral in the proof of our theorem is the following result on hyperinterpolation on spheres proven by Reimer in \cite{Reimer00}.

\begin{thm}[Theorem 1, \cite{Reimer00}] 
	If $\CL_{2n}$ is the hyperinterpolation operator arising from a cubature formula of degree $2n$ with positive weights on $S^{d-1}$, then the norm of $\CL_{2n}$ as a projections operator from $C(S^{d-1})$ to $\Pi_{2n}$ satisfies
	\begin{displaymath}
		a_d n^{(d-2)/2}\leq \|\CL_{2n}\| \leq b_d n^{(d-2)/2},
	\end{displaymath}
	for positive constants $a_d$, $b_d$.
	\label{thm:reimer}
\end{thm}

Before we prove Theorem \ref{thm:main}, we first provide two lemmas that illustrate the relationship betwee cubature formulae on $B^d$ with respect to the weight $\wb_{m/2}$ and cubature formulae on $S^{d+m}$.  Both of these lemmas rely on the integration formula below.

\begin{lem}[Lemma 3.8.9, \cite{DX}]
	Suppose that $d$ and $m$ are positive integers.  Then
	\begin{equation}
		\label{eqn:dxballsphere}
		\int_{S^{d+m}}\, f(y)\, d\omega(y)=\int_{B^d} \wb_{m/2}(\xb)\int_{S^m} \, f(\xb, \sqrt{1-\|\xb\|^2} \eta)\, d\omega(\eta) d\xb.
	\end{equation}
	\label{lem:dxballsphere}
\end{lem}

Our first lemma shows that we may use a cubature formula on $S^{d+m}$ to obtain a cubature formula on $B^d$ with respect to $\wb_{m/2}$.

\begin{lem}
	\label{lem:spheretoballcubature}
	Suppose that $\sum_{\alpha \in A} \lambda_\alpha f(x_\alpha)$ is a cubature formula of degree $n$ on $S^{d+m}$.  This cubature formula is also of degreen $n$ on $B^d$ with respect to the weight function $\wb_{m/2}$, in the sense that
	\begin{displaymath}
		\int_{B^d}\, f(\xb) \, \wb_{m/2}\, dx = \sum_{\alpha \in A} \lambda_\alpha \tilde{f}(x_\alpha),
	\end{displaymath}
	whenever $f$ is a polynomial of degree less than or equal to $n$, where $\tilde{f}$ is a function defined on $S^{d+m}$, but only depending on the first $d$ variables, and satisfying
	\begin{displaymath}
		\tilde{f}(\xb, \xb^c) = \frac{1}{|S^m|} f(\xb),
	\end{displaymath}
	where $|S^m|$ is the surface measure of $S^m$.
\end{lem}

\begin{proof}
	First observe that 
	\begin{align*}
		\int_{B^d} \, f(\xb) \wb_{m/2}(\xb) \, d\xb & =\frac{1}{|S^m|} \int_{S^m} \int_{B^d} \tilde{f}(\xb, \eta \sqrt{1-\xb^2}) \wb_{m/2}(\xb)\, d\xb\, d\omega_m(\eta)\\
		&=\frac{1}{|S^m|} \int_{S^{d+m}} \tilde{f}(\yb) d\omega_{d+m}(\yb),
	\end{align*}
	where we have used Lemma \ref{lem:dxballsphere} for the second equality.   Applying the cubature formula to the last integral immediately gives the result. 
\end{proof}

The next lemma illustrates how a cubature formula on $B^{d}$ with respect to $\wb_{m/2}$ can be used to obtain a cubature formula on $S^{d+m}$.

\begin{lem}
	Suppose that $\sum_{\alpha \in A} \lambda_\alpha f(x_\alpha)$ is a cubature formula of degree $n$ on $B^d$ with respect to $\wb_{m/2}(\xb$).  Let $\sum_{\beta \in B} \nu_\beta f(t_\beta)$ be a cubature formula of degree $n$ on $S^m$.  Then the cubature formula 
	
	\begin{displaymath}
		\sum_{\alpha \in A, \beta \in B} \lambda_\alpha \nu_\beta f(x_\alpha, \sqrt{1-\|x_\alpha\|^2} t_\beta)
	\end{displaymath}
	is a cubature formula of degree $n$ on $S^{d+m}$.
	\label{lem:cubatureballtosphere}
\end{lem}

\begin{proof}
	Let $f$ be a polynomial of degree $n$ on $\RR^{d+m+1}$.  Note that, in \eqref{eqn:dxballsphere}, any odd powers of $\sqrt{1-\|\xb\|^2}$ will vanish after integrating over $S^m$, since an odd power of $\sqrt{1-\|\xb\|^2}$ in the integrand must be multiplied by at least one odd power of a component of $\eta$, and any odd component of $\eta$ will result in zero upon integrating.  Hence, after integrating over $S^m$, we are left with a polynomial of degree less than or equal to $n$ in $\xb$.  
	
	It follows that
	\begin{align*}
		\sum_{\alpha \in A} \sum_{\beta \in B} \lambda_\alpha \nu_\beta f(x_\alpha, \sqrt{1-\| x_\alpha\|^2}t_\beta) &= \sum_{\alpha \in A} \lambda_\alpha \int_{S^m} \, f(x_\alpha, \sqrt{1-\|x_\alpha\|^2} \eta) d\omega_m(\eta)\\
		&= \int_{B^d} \int_{S^m} \, f(\xb, \sqrt{1-\|\xb\|^2} \eta) d\omega_m(\eta) \\
		& \phantom{AAAAAAAAAAAA}\times \wb_{m/2}(\xb) d\xb\\
		& = \int_{S^{d+m}} f(\yb) \omega_{d+m}(\yb),
	\end{align*}
	so that the extension is a degree $n$ cubature on $S^{d+m}$.
\end{proof}

We now prove Theorem \ref{thm:main}.

\begin{proof}
We first use the standard inequality for estimating Lebesgue constants, and then apply Corollary \ref{cor:repkerballtosphere} to obtain
\begin{align*}
	\|\CL_{2n}\| & \leq \sup_{\xb \in B^d} \sum_{\alpha \in A} \lambda_\alpha \left| K_n \left( \wb_{\mu}; \xb, x_\alpha \right) \right|\\
	& \leq \sup_{\xb \in B^d} \sum_{\alpha \in A} \lambda_\alpha \left| \int_{S^m} K_n \left(1; (\xb,\xb^C), (x_\alpha, \sqrt{1-\|x_\alpha\|^2}\eta)  \right) d \omega_m(\eta) \right|,
\end{align*}
where we have let $\mu=2m$.  Let $\displaystyle \sum_{\beta \in B} \nu_\beta f(t_\beta)$ be any cubature formula of degree $2n$ for $S^{m}$.  Since the integrand is a polynomial of degree $n$ in $\eta$, we may apply a cubature formula of degree $2n$ to the integral over $S^m$ and obtain
\begin{align}
	\label{ali:junk}
	\| \CL_{2n} \| & \leq \sup_{\xb \in B^d} \sum_{\alpha \in A} \lambda_\alpha \left| \sum_{\beta \in B} \nu_\beta K_n\left( 1; (\xb, \xb^C), (x_\alpha, \sqrt{1-\|x_\alpha\|^2} t_\beta ) \right)\right| \notag\\
	& \leq \sup_{\yb \in S^{d+m}} \sum_{\alpha \in A} \sum_{\beta \in B} \lambda_\alpha \nu_\beta \left| K_n\left( 1; \yb, (x_\alpha, \sqrt{1-\|x_\alpha\|^2} t_\beta ) \right)\right|.
\end{align}
By Lemma \ref{lem:cubatureballtosphere}, the cubature formula $\sum_{\alpha, \beta} \lambda_\alpha \nu_\beta f(x_\alpha, \sqrt{1-\|x_\alpha\|^2}t_\beta)$ is a cubature formula of degree $2n$ on $S^{d+m}$.  In the proof of Theorem \ref{thm:reimer}, it was proven that the expression in \eqref{ali:junk} is less than $b_{d+m} (2n)^{(d+m-1)/2}$; hence we immediately obtain
\begin{align*}
	\|\CL_{2n}\| & \leq c_{d+m} n^{(d+m-1)/2} 
\end{align*}
for a positive constant $c_{d+m}$.

To finish the proof, we again let $\sum \nu_\beta f(t_\beta)$ be a cubature formula of degree $2n$ on $S^m$, and observe that if $f \in C(B^{d})$ with $\|f\|_{C(B^d)} \leq 1$, then we may define $\tilde{f}$ on $S^{d+m}$ as in Lemma \ref{lem:spheretoballcubature}, and obtain
\begin{align*}
	\sum_{\alpha \in A} \lambda_\alpha f(x_\alpha) & = \sum_{\alpha, \beta} \lambda_\alpha \nu_\beta \tilde{f}(x_\alpha, \sqrt{1-\|x_\alpha\|^2}t_\beta).
\end{align*}
However, by Lemma \ref{lem:cubatureballtosphere}, the cubature formula on the right of the equality is a cubature formula of degree $2n$ on $S^{d+m}$, and $\tilde{f}$ is in $C(S^{d+m})$ and has uniform norm less than or equal to $|S^m|^{-1}$.  Hence, by Theorem \ref{thm:reimer}, there is a positive constant $b_{d,m}$ satisying
\begin{align*}
	\left|\sum_{\alpha \in A} \lambda_\alpha f(x_\alpha) \right| & = \left|\sum_{\alpha, \beta} \lambda_\alpha \nu_\beta \tilde{f}(x_\alpha, \sqrt{1-\|x_\alpha\|^2}t_\beta)\right|\\
	&\geq b_{d,m}n^{(d+m-1)/2},
\end{align*}
which completes the proof of Theorem \ref{thm:main}.
\end{proof}

An immediate result of this theorem is an improvement on the Lebesgue constant estimate in \cite{Atkinson09}.

\begin{cor}  Let $\CL_n$ be a hyperinterpolation operator of degree $n$ on $B^2$ with respect to the standard Lebesgue measure.  Then there exist positive constants $b$ and $c$ satisfying $b\, n \leq \|\CL_n\| \leq c\, n$.
\end{cor}


\begin{thebibliography}{99}

\bibitem{Atkinson09}
	Olaf Hansen, Kendall Atkinson, David Chien,
	On the norm of the hyperinterpolation operator on the unit disk and its use for the solution of the nonlinear Poisson equation,
	\textit{IMA J Numer Anal} \textbf{29} (2009), 257-283.

\bibitem{DX}
        C. F. Dunkl, Yuan Xu,
       \textit{Orthogonal polynomials of several variables},
        Cambridge Univ. Press, 2001. 

\bibitem{HesseSloan}
	Kerstin Hesse, Ian H. Sloan,
     	Hyperinterpolation on the sphere,
	\textit{Frontiers in interpolation and approximation}, Pure Appl. Math. (Boca Raton), \textbf{282}, (2007), 213-248.

\bibitem{LeGiaSloan}
	T. Le Gia, I. H. Sloan,
	The Uniform Norm of Hyperinterpolation on the Unit Sphere on an Arbitrary Number of Dimensions,
	\textit{Constr. Approx.}, \textbf{17}, (2001), 249-265.

\bibitem{LuoMeng}
	Zhongxuan Luo, Zhaoliang Meng,
	Cubature formulas over the $n$-sphere,
	\textit{J. Comput. Appl. Math.} \textbf{202} (2007), 511-522.

\bibitem{Reimer00}
        M. Reimer, 
        Hyperinterpolation on the Sphere at the Minimal Projection Order, 
        \textit{J. Approx. Theory} \textbf{104} (2000), 272-286. 

\bibitem{Sloan95}
	Ian Sloan, 
     	Polynomial interpolation and hyperinterpolation over general regions,
   	\textit{J. Approx. Theory} \textbf{83}, (1995), 238-254.

\bibitem{Stroud}
	A. H. Stroud,
	\textit{Approximate calculation of multiple integrals},
	Prentice-Hall Inc., 1971.

\bibitem{Szego}
	G\'{a}bor Szeg\"{o},
	\textit{Orthogonal Polynomials},
	American Mathematical Society, 1975.


\bibitem{Xu01}
	Yuan Xu,
	Orthogonal polynomials and cubature formulae on balls, simplices, and spheres,
    	\textit{J. Comput. Appl. Math.} \textbf{127} (2001), 349-368.
\end{thebibliography}
\end{document}